\documentclass[11pt]{amsart}

\usepackage[top=3.2cm,bottom=2.7cm,left=2.7cm,right=2.7cm]{geometry}

\usepackage{amsmath}
\usepackage{amssymb}
\usepackage{amsthm}

\newtheorem{thm}{Theorem}
\newtheorem{lemma}[thm]{Lemma}
\newtheorem{cor}[thm]{Corollary}

\DeclareMathOperator\erf{erf}

\newcommand{\E}{\mathbb{E}}

\newcommand{\be}{\begin{equation}}
\newcommand{\ee}{\end{equation}}

\newcommand{\C}{\mathbb{C}}

\newcommand{\EE}{\mathbb{E}}

\newcommand{\ol}{\overline}

\newcommand{\CC}{\mathbb{C}}

\renewcommand{\P}{\mathbb{P}}

\theoremstyle{remark}

\usepackage{xcolor}

\usepackage{mathtools} 
\mathtoolsset{showonlyrefs,showmanualtags} 

\title[The arc length of a random lemniscate]{The arc length of a random lemniscate}
\author{Erik Lundberg and Koushik Ramachandran}
\date{}
\begin{document}


\begin{abstract}
A polynomial lemniscate is a curve in the complex plane 
defined by $\{z \in \C:|p(z)|=t\}$. 
Erd\"os, Herzog, and Piranian 
posed the extremal problem of determining the 
maximum length of a lemniscate $\Lambda=\{z \in \mathbb{C}:|p(z)|=1\}$ 
when $p$ is a monic polynomial of degree $n$. 
In this paper, we study the 
length and topology of a random lemniscate 
whose defining polynomial has independent Gaussian coefficients.
In the special case of the Kac ensemble
we show that the length approaches a nonzero constant as $n \rightarrow \infty$.
We also show that the average number of connected components 
is asymptotically $n$,
and we observe a positive probability (independent of $n$)
of a giant component occurring.
\end{abstract}

\maketitle

\section{Introduction}
A (polynomial) lemniscate is a curve defined in the complex plane by the equation $|p(z)| = t$, 
where $p$ is a polynomial.
If the degree of $p$ is $n,$ then from the conjugation-invariant equation $p(z) \overline{p(z)} = t^2$,
it is apparent that the lemniscate is a real algebraic curve of degree $2n$.
Calculating the length of a lemniscate is a problem of classical Mathematics 
that played a role in the development of elliptic integrals.
Namely, the length of Bernoulli's lemniscate $|z^2-1| = 1$
is an elliptic integral of the second kind
(the same type of integral that appears in classical mechanics,
as the period of a pendulum, and in classical statics, as the length of an elastica).

\subsection{The Erd\"os lemniscate problem} 
Erd\"os, Herzog, and Piranian \cite{Erdos} 
posed the extremal problem  of determining the 
maximum length of a lemniscate
\be \Lambda=\left\{z\in \CC\,:\, |p(z)| = 1 \right\}\ee
when $p$ is a monic polynomial of degree $n$.
The problem was restated by Erd\"os several times (e.g., see \cite{Erdos2})
and is often referred to as the \emph{Erd\"os lemniscate problem}.
Taking $p$ monic guarantees that the length of the lemniscate is bounded, 
for instance by $2\pi n$ \cite{Danchenko}.
The maximum was conjectured \cite{Erdos} to occur for the 
so-called \emph{Erd\"os lemniscate}, i.e, when $p(z) = z^n-1$.
This conjecture remains open
but has seen positive results \cite{Borwein, EremHay, KuTk},
and Fryntov and Nazarov \cite{FryntovNazarov} have proved 
that Erd\"os lemniscate is indeed 
a \emph{local} maximum and that as $n \rightarrow \infty$ 
the maximum length is $2n + o(n)$ which is asymptotic to the conjectured extremal.

 \begin{figure}
\begin{center}
\includegraphics[width=0.4\textwidth]{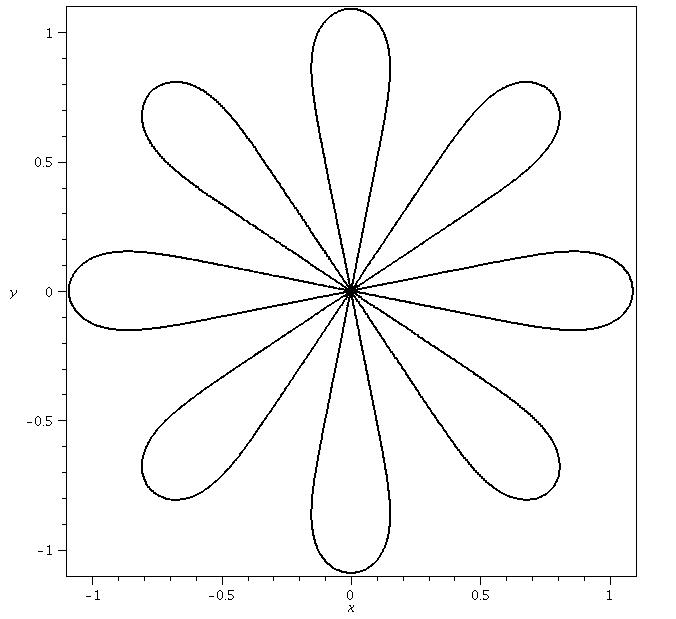}
\caption{\label{fig:ErdosLemni} The Erd\"os lemniscate  for $n=8$.}
\end{center}
\end{figure}

\subsection{The arc length of a random lemniscate}
A random variable X has the 
standard complex Gaussian distribution 
if it has density $\frac{1}{\pi}\exp(-|z|^2)$ on $\mathbb{C}.$ 
We denote this by $X\sim N_{\CC}(0,1).$

\vspace{0.1in}

\noindent Motivated by seeking a broad point of view on the Erd\"os lemniscate problem,
we give a probabilistic treatment of the length,
by studying the \emph{average} outcome for a random polynomial lemniscate. 
We select $\Lambda = \Lambda_n$ randomly by taking $p_n(z)$ to be a random polynomial 
from the Kac ensemble,
\be\label{eq:KacModel}
p_n(z)=\sum_{k=0}^{n}a_{k}z^k ,\ee
where $a_{k} \sim N_{\CC}(0,1)$ are independent, 
identically distributed complex Gaussians. The resulting distribution for the random curve $\Lambda$ is invariant under rotation of the angular coordinate.
Indeed, we have:
\be 
|p_n(e^{i\theta }z)|=\left|\sum_{k=0}^{n}a_{k}e^{i k \theta}z^k\right|,
\ee
 and invariance follows from the observation that $b_{k}= a_{k}e^{i k \theta}$ are i.i.d and distributed as $N_{\CC}(0,1)$.

 \begin{figure}[t]
\begin{center}
\includegraphics[width=0.31\textwidth]{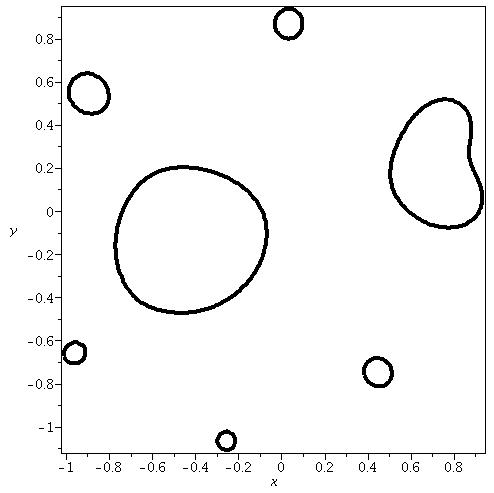}
\includegraphics[width=0.32\textwidth]{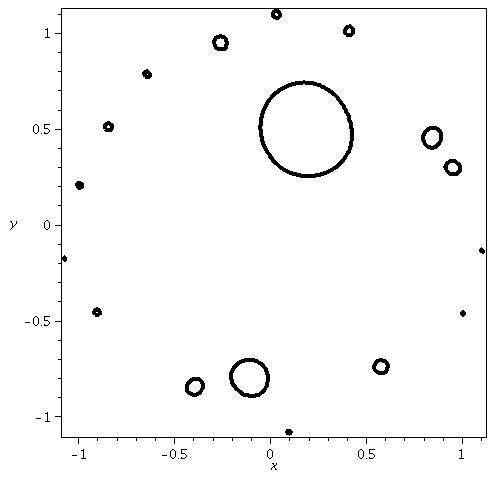}
\includegraphics[width=0.335\textwidth]{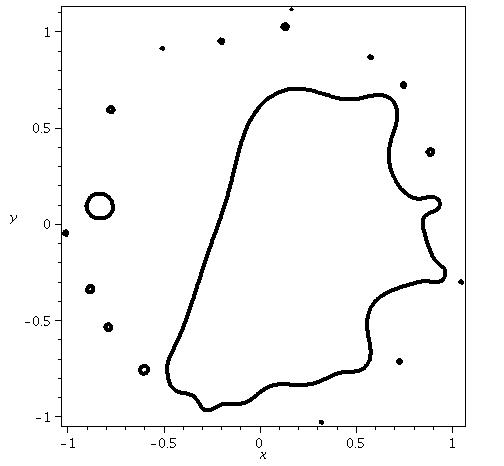}
\includegraphics[width=0.315\textwidth]{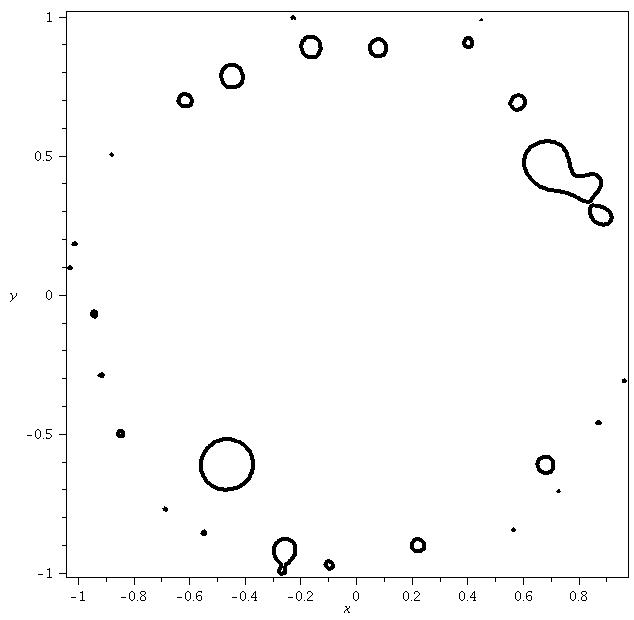}
\includegraphics[width=0.315\textwidth]{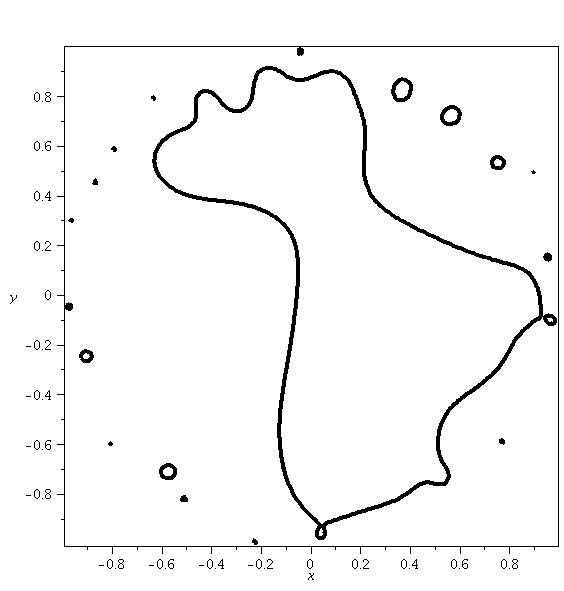}
\includegraphics[width=0.31\textwidth]{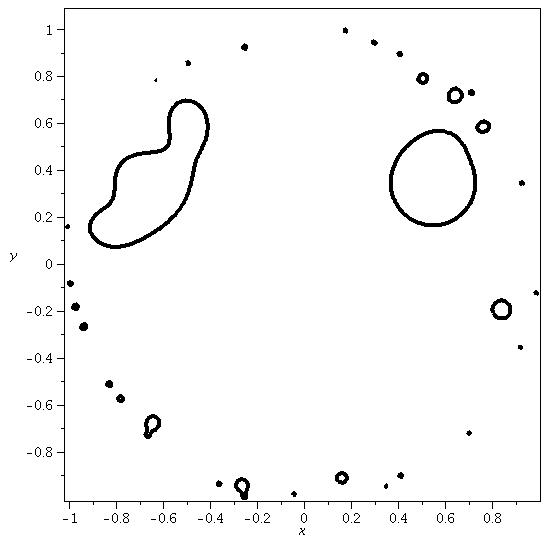}
\caption{\label{fig:samples} Random lemniscates using Kac polynomials of degree 
$n=10, 20, 30, 40, 50, 60$ (from left to right).}
\end{center}
\end{figure}

We now state our main result.
 
 \begin{thm}\label{thm:Kac}
Consider a sequence of random polynomials $p_n(z) = \sum_{k=0}^{n}a_kz^k,$ where the $a_k$ are i.i.d $N_{\mathbb{C}}(0, 1).$ Let $\Lambda_n = \left\{z\in \CC\,:\, |p_n(z)| = 1 \right\}.$ Then, $$\lim_{n \rightarrow \infty} \EE |\Lambda_n| = C ,$$
 where the constant $C \approx 8.3882$ is given by 
 the integral \eqref{eq:limit} below.
 \end{thm}

\subsection{The Erd\"os lemniscate is an outlier}\label{sec:outlier}

The following Corollary of Theorem \ref{thm:Kac}
provides weak concentration of measure around lemniscates
having length of constant order.
 \begin{cor}\label{cor:outlier}
Let $L_n$ be any sequence with $L_n \rightarrow \infty$ as $n \rightarrow \infty$.
The probability that $|\Lambda_n| \geq L_n$ converges to zero.
 \end{cor}

\begin{proof}
Since the length $|\Lambda_n|$ is a positive random variable, 
we can apply Markov's inequality:
$$ \P \{ |\Lambda_n| \geq L_n \} \leq \frac{\EE |\Lambda_n|}{L_n} =O(L_n^{-1}), \quad \text{ as } n \rightarrow \infty ,$$
by Theorem \ref{thm:Kac}.
\end{proof}

In particular, the probability that the length has the same order as the extremal case
(i.e., exceeding some fixed portion of $n$) converges to zero.

\subsection{The connected components of a random lemniscate}
How many connected components does a random lemniscate have? 
This question was addressed in \cite{Lemni} 
in the setting of rational lemniscates. 
The next theorem answers this question for a random polynomial lemniscate based
on the Kac model.  The notation $b_0(\Lambda_n)$ denotes the zeroth Betti number,
which is the number of connected components.

\begin{thm}\label{thm:cc}
The number $b_0(\Lambda_n)$ of connected components of a random
Kac lemniscate satisfies
$$\E b_0(\Lambda_n) \sim n, \quad \text{as } n \rightarrow \infty.$$
\end{thm}

Along with Theorem \ref{thm:Kac},
this indicates a prevalence of small components.
In fact,
the idea of the proof of Theorem \ref{thm:cc}
is to check in the vicinity of a zero
for a component to appear
within a disk of radius $n^{-1-\alpha}$
where $0<\alpha<1/2$.
This suggests that relatively few
components account for most of the length.
It seems natural to further investigate the distribution of lengths of components,
and we begin to do this with the next Thereom that establishes,
with some positive probability independent of n, 
the presence of at least one ``giant component'' 
(compare with the samples plotted in Figure 2).
 
 \begin{thm}\label{thm:giant}
 Fix $r\in (0,1)$ and let $\Lambda_n$ be a random Kac lemniscate. There is a positive probability (depending on $r$ but independent of $n$) that $\Lambda_n$ has a component with length at least $2\pi r$.
 \end{thm}



 
\subsection{Remarks}

The Erd\"os lemniscate is 
extremely singular and symmetric (see Figure \ref{fig:ErdosLemni}),
and its length appears to diminish rapidly under perturbations.
Naively, this suggests that
it occupies a rather far corner of the parameter space.
The probabilistic approach taken here provides a framework 
for making this notion precise
as we have done in Section \ref{sec:outlier}.
The authors expect that the rate of decay in Corollary \ref{cor:outlier} 
can be improved,
and it would be interesting to investigate this topic
from the point of view of large deviations.

The outcome for the average length of a random lemniscate
depends on the definition of ``random''.
The Kac ensemble is one of the most well-studied instances,
and it seems especially appropriate 
in the context of the Erd\"os lemniscate problem,
since the zeros of $p_n$ resemble those of the defining polynomial of the Erd\"os lemniscate
in that they are approximately equidistributed on the unit circle \cite{SZ, ZZ}.
We consider several models in the sections below,
including the case that the variances have binomial coefficient weights
and also the case in which they have reciprocal binomial coefficient weights.
In each of these cases, the expected length has order $O(n^{-1/2})$.

Another extremal problem, to find the maximal spherical length of a rational lemniscate,
was posed and solved by Eremenko and Hayman \cite{EremHay}. 
Lerario and the first author
considered random rational lemniscates on the 
Riemann sphere \cite{Lemni}
and computed the average spherical length.
They also studied the connected components
while giving special attention to
nesting of components,
which can occur for rational lemniscates,
but is not possible 
for polynomial lemniscates
(the latter statement follows from the maximum principle).

\subsection{Outline of the paper}
Theorem \ref{thm:Kac} will follow from a more general result
proved in Section \ref{sec:general},
namely, Theorem \ref{thm:general} 
provides the expected length while allowing 
the coefficients appearing in \eqref{eq:KacModel}
to be independent centered Gaussians with different variances.
The methods in proving Theorem \ref{thm:general} 
are based on planar integral geometry
combined with the Kac-Rice formula.
In Section \ref{sec:Kac}, 
we then derive Theorem \ref{thm:Kac} as a consequence
of Theorem \ref{thm:general}.
We also apply Theorem \ref{thm:general} to three other models:
lemniscates generated by Kostlan polynomials are treated in Section \ref{sec:Kostlan},
Weyl polynomials in Section \ref{sec:Weyl},
and a model that we call the ``reciprocal binomial'' 
model is considered in Section \ref{sec:reciprocal}.
Returning to the Kac model in Section \ref{sec:components},
we study the connected components of a random lemniscate;
we prove Theorem \ref{thm:cc} in Section \ref{sec:cc}
and Theorem \ref{thm:giant} in Section \ref{sec:giant}.


\section{A length formula for Gaussian polynomials}\label{sec:general}

In this section we assume that the 
coefficients appearing in $p_n(z)$ 
are centered, independent, but not necessarily identically distributed complex Gaussians.

\subsection{Length and integral geometry}

Applying the integral geometry formula as in \cite{EremHay}, we have:
$$|\Lambda_n | = \frac{1}{2} \int_0^\pi \int_{-\infty}^\infty N_n(\theta,y) d \theta dy ,$$
where $N_n(\theta,y)$ is the number of intersections of $\Lambda_n$ with the line $L(\theta,y) := \{z \in \CC: \Im (e^{-i\theta}z) = y \}$.
Taking the expectation of both sides and using the rotational invariance of $\Lambda_n$, we have:
\begin{equation}\label{eq:IGF}
\EE |\Lambda_n | = \frac{1}{2} \int_0^\pi \int_{-\infty}^\infty \EE N_n(\theta,y) d \theta dy = \frac{\pi}{2} \int_{-\infty}^\infty \EE N_n(0,y) dy.
\end{equation}

\subsection{The Kac-Rice formula}
We use the Kac-Rice formula to compute $\EE N_n(0,y)$ 
which equals the average number of real zeros of the function
$$ p_n(z) \ol{p_n(z)} - 1,$$
restricted to the line $L(0,y)$.
We have:
$$ \frac{\partial}{\partial x} ( p_n(z) \ol{p_n(z)} - 1 ) = p_n'(z) \ol{p_n(z)} + p_n(z) \ol{p_n'(z)} .$$ 

Applying the Kac-Rice formula, we have:
\begin{equation}\label{eq:KR}
\EE N_n(0,y) = \int_{-\infty}^\infty \EE \delta(|p_n(z)|^2-1) |p_n'(z) \ol{p_n(z)} + p_n(z) \ol{p_n'(z)} |  dx.
\end{equation}

For the sake of notational clarity we will henceforth suppress the dependence on $n.$ So for instance $\Lambda_n$ will be denoted by $\Lambda,$  $p_n$ by $p$ etc. We can rewrite \eqref{eq:KR} in terms of the Gaussian random complex vector 
$(U,V) = (p(z),p'(z))$ whose joint probability density function is:
\begin{equation}\label{eq:density}
\rho(u,v;x+iy) = \frac{1}{\pi^2 |\Sigma|} \exp \{- (u,v)^* \Sigma^{-1} (u,v) \},
\end{equation}
where $\Sigma$ is the covariance matrix of $(U,V) = (p(z),p'(z))$,
which can be computed explicitly using the covariance kernel $K(z,w)$:
$$K(z,w) = \EE p(z) \ol{p(w)}.$$
Namely, we have:
$$ \Sigma = \left( \begin{array}{cr}
a & b \\ \bar{b} & c 
\end{array} \right) ,$$
where
\begin{equation}\label{eq:abc}
a = K(z,z), b = \partial_{z} K(z,z)\hspace{0.05in}\mbox{and,}\hspace{0.05in} c  = \partial_{z} \partial_{\bar{z}} K(z,z).
\end{equation}

In terms of this joint density, the expectation inside \eqref{eq:KR}
can be expressed as:
\begin{align*}
 \EE \delta(|p(z)|^2-1) |p'(z) \ol{p(z)} + p(z) \ol{p'(z)} |  &= \int_{\C} \int_{\C} \delta(|u|^2-1) |v\bar{u} + u \bar{v}| \rho(u,v;x+iy) dA(v) dA(u) \\
 &= \int_{|u|=1} \int_{\C} \frac{1}{2|u|}  |v\bar{u} + u \bar{v}| \rho(u,v;x+iy) dA(v) dA(u)  \\
 &= \frac{1}{2} \int_{|u|=1} \int_{\C} |v\bar{u} + u \bar{v}| \rho(u,v;x+iy) dA(v) dA(u),
 \end{align*}
 where we have used the composition property of the $\delta$-function
 (\cite{Hor}, Chapter $6$) allowing integration against $\delta(|u|^2-1)$
 to be replaced by an integration along the set $|u|^2=1$.
 
For $|u|=1$, we notice that 
\begin{align*}
\rho(u,v;z) &= \frac{1}{\pi^2 |\Sigma|} \exp \{- u\bar{u}(1,\bar{u}v)^* \Sigma^{-1} (1,\bar{u}v) \} \\
&= \frac{1}{\pi^2 |\Sigma|} \exp \{- (1,\bar{u}v)^* \Sigma^{-1} (1,\bar{u}v) \} \\
&= \rho(1,\bar{u}v;z).
\end{align*}
Making the change of variables $t = \bar{u}v$, $dA(t) = dA(v)$,
the integral above becomes
 \begin{equation}
 \frac{1}{2} \int_{|u|=1} \int_{\C} |t + \bar{t}| \rho(1,t;x+iy) dA(t) du = \pi \int_{\C} |t + \bar{t}| \rho(1,t;x+iy) dA(t) . 
\end{equation}

Thus, we have:
$$\EE N(0,y) = 2\pi \int_{-\infty}^\infty \int_{\C} |\Re\{t\}| \rho(1,t;x+iy) dA(t) dx.$$

Inserting this into the integral geometry formula \eqref{eq:IGF} gives:
\begin{equation}\label{eq:reduced}
\EE |\Lambda | =  \pi^2 \int_{\C} \int_{\C} |\Re \{t\} | \rho(1,t;z) dA(t) dA(z).
\end{equation}

Observe that the density $\rho$ can be factored:
\begin{align*}
\rho(1,t;z) &= 
\frac{\exp \{-\frac{1}{a} \}}{\pi a} \frac{a}{\pi |\Sigma|} \exp \left\{-\frac{a}{|\Sigma|} \left| t-\frac{b}{a} \right|^2 \right\} \\
&= \frac{\exp \{-\frac{1}{a} \}}{\pi a}  \hat{\rho}(t),
\end{align*}
where $\hat{\rho}$ is the probability density function for a complex Gaussian
$N_\C(\mu,\sigma^2)$ with mean $\mu = b/a$ and variance $\sigma^2 = \frac{|\Sigma|}{a}$.
Thus, the following lemma applies.

\begin{lemma}\label{lemma:absreal}
Let $\zeta \sim N_\C(\mu,\sigma^2)$
be a complex Gaussian with mean $\mu = \mu_1 + i \mu_2$.
Then the absolute moment $\EE |\zeta_1|$ of the real part of $\zeta = \zeta_1 + i \zeta_2$ is given by
$$\EE |\zeta_1| = \frac{\sigma}{\sqrt{\pi}} \exp\{-\mu_1^2/\sigma^2\} + |\mu_1| \erf(|\mu_1|/\sigma) .$$
\end{lemma}

\begin{proof}[Proof of Lemma \ref{lemma:absreal}]
We have
 \begin{align*}
 \EE |\zeta_1|  &= \frac{1}{\pi \sigma^2} \int_{\C} |\zeta_1| \exp\left\{\frac{-|\zeta-\mu|^2}{\sigma^2}\right\} dA(\zeta) \\
  &=  \frac{1}{\pi} \int_{\CC} |\sigma w_1 + \mu_1| \exp\left\{ -|w|^2 \right\} dA(w),
 \end{align*}
where we have made the change of variables $w = \frac{\zeta-\mu}{\sigma}$, $dA(w) = \frac{1}{\sigma^2}dA(\zeta)$.

Letting $H:=\{w \in \C:\sigma w_1 + \mu_1 > 0 \}$, we can rewrite the above integral as:
$$ \frac{1}{\pi}\left( \int_{H} (\sigma w_1 + \mu_1) \exp\{ -|w|^2 \} dw_1 dw_2 - \int_{\C \setminus H} (\sigma w_1 + \mu_1) \exp\{ -|w|^2 \} dw_1 dw_2 \right).$$
Since $\sigma w_1$ is odd and $\mu_1$ is even (with respect to $w_1$) this can be rewritten as:
\begin{equation}\label{eq:first}
\frac{1}{\pi}\left( \int_{R} |\mu_1| \exp\{ -|w|^2 \} dw_1 dw_2 + \sigma \int_{\C \setminus R} |w_1|  \exp\{ -|w|^2 \} dw_1 dw_2 \right),
\end{equation}
where $R:=\left\{w \in \C: |w_1| < \frac{|\mu_1|}{\sigma}  \right\}$.
The first integral can be computed in terms of the error function, $\erf$:
\begin{equation}\label{eq:second}
\int_{R} |\mu_1| \exp\{ -|w|^2 \} dw_1 dw_2 = \pi |\mu_1| \erf(|\mu_1|/\sigma),
\end{equation}
and the second integral is elementary:
\begin{equation}\label{eq:third}
\int_{\C \setminus R} |w_1|  \exp\{ -|w|^2 \} dw_1 dw_2 = \sqrt{\pi}\exp\{-\mu_1^2/\sigma^2\}.
\end{equation}
Collecting \eqref{eq:first}, \eqref{eq:second}, and \eqref{eq:third}, 
we arrive at the formula stated in the lemma. 
\end{proof}

Applying Lemma \ref{lemma:absreal} to \eqref{eq:reduced}, we obtain the following
main result of this section:
\begin{thm}\label{thm:general}
Let $p(z)$ be a random polynomial whose coefficients are independent centered Complex Gaussians. 
Then the expected length of its lemniscate $\Lambda := \{z \in \CC : |p(z)| = 1 \}$ is given by
\begin{equation}\label{eq:nonasymp}
\EE |\Lambda | =  \sqrt{\pi} \int_{\C} \frac{\exp \{-\frac{1}{a} \}}{a} \left[ \sqrt{\frac{|\Sigma|}{a}}\exp\left\{-\frac{|\Re{b}|^2}{a|\Sigma|} \right\} + \sqrt{\pi} \frac{|\Re b|}{a} \erf \left\{ |\Re b|/\sqrt{a |\Sigma|} \right\} \right] dA(z).
\end{equation}
where as above $|\Sigma|$ denotes the determinant of the covariance matrix
$\Sigma$ and, the terms $a, b, c$ are
the entries of $\Sigma$ given by \eqref{eq:abc}.
\end{thm}

\section{Kac polynomials: proof of Theorem \ref{thm:Kac}}\label{sec:Kac}

In the case $p(z)$ is a random Kac polynomial, for the entries in the covariance matrix,
$$ \Sigma = \left( \begin{array}{cr}
a & b \\ \bar{b} & c 
\end{array} \right) ,$$
we have
\begin{align*}
a &= K(z,z) = \sum_{k=0}^n |z|^{2k},\\
b &= \partial_{z} K(z,z) = \bar{z} \sum_{k=1}^n k |z|^{2k-2},\\
c &= \partial_{z} \partial_{\bar{z}} K(z,z) = \sum_{k=1}^n k^2 |z|^{2k-2}.
\end{align*}

We will show that the pointwise limit of the integrand appearing in \eqref{eq:nonasymp} as $n\to\infty$ is:
\begin{equation}\label{eq:pointwise}
\left\{ \begin{array}{cc}
  \exp\left\{ -(1-|z|^2) \right\}\left[\frac{\exp\left\{-x^2(1-|z|^2)\right\} }{(1-|z|^2)^{1/2}} + \sqrt{\pi} x\erf\left\{x\sqrt{1-|z|^2}\right\} \right], \quad |z| < 1, \quad   \\
 0 , \quad |z| \geq 1.
\end{array} \right.
\end{equation}

We will also show that the dominated convergence theorem applies,
so that the integral in Theorem \ref{thm:general}
has a limit as $n \rightarrow \infty$
given by the integral of \eqref{eq:pointwise}.
After changing to polar coordinates, this becomes:
\begin{align}\label{eq:limit}
C &:= \lim_{n \rightarrow \infty} \EE |\Lambda | \\
&= \sqrt{\pi} \int_{|z|<1} \exp\left\{ -(1-|z|^2) \right\}\left[ \frac{\exp\left\{-x^2(1-|z|^2)\right\} }{(1-|z|^2)^{1/2}} + \sqrt{\pi} x\erf\left\{x\sqrt{1-|z|^2}\right\} \right] dA(z)\\
&\approx 8.3882,
\end{align}
which proves Theorem \ref{thm:Kac}.  
It remains to compute the pointwise limit and to show dominated convergence.

First, we derive certain formulas from the covariance kernel $K(z,w)$ of the Kac polynomial,
$$K(z,w) = \EE p(z) \ol{p(w)} = \sum_{k=0}^n (z \bar{w})^k = \frac{1-(z \bar{w})^{n+1}}{1-z\bar{w}}.$$
Notice that
\begin{equation}\label{eq:a}
a = K(z,z) = \sum_{k=0}^n |z|^{2k} = \frac{1}{1-|z|^2}- \frac{|z|^{2n+2}}{1-|z|^2}.
\end{equation}

We have
\begin{equation}\label{eq:b}
\frac{\Re \{b\}}{a} = \Re \{ \partial_z \log K(z,z) \}= 
\frac{x}{(1-|z|^2)} - \frac{(n+1) x |z|^{2n}}{1-|z|^{2n+2}},
\end{equation}
and from this we observe that for $|z| < 1,$
\begin{align*}\label{eq:bound2}
\frac{\Re \{b\}}{a^2} &= 
x\frac{\frac{1}{1-|z|^2}-\frac{(n+1)|z|^{2n}}{1-|z|^{2n+2}}}{\frac{1}{1-|z|^2}-\frac{|z|^{2n+2}}{1-|z|^{2}}}\\
&= x\frac{1-(n+1)\frac{|z|^{2n}}{\sum_{k=0}^n|z|^{2k}}}{1-|z|^{2n+2}}\\
&\leq x,
\end{align*}
and as $n\to\infty,$ $\frac{\Re \{b\}}{a^2}\to x.$

On the other hand for $|z| > 1,$ we note that
\begin{align*}
\frac{\Re \{b\}}{a^2} &= x\frac{\frac{(n+1)|z|^{2n}}{|z|^{2n+2}-1}-\frac{1}{|z|^2-1}}{\frac{|z|^{2n+2}}{|z|^2-1}-\frac{1}{|z|^2-1}}\\
& = x\frac{\frac{(n+1)|z|^{2n}}{\sum_{k=0}^n|z|^{2k}}-1}{|z|^{2n+2}-1}\\
&\leq x,
\end{align*}
and as $n\to\infty,$ $\frac{\Re \{b\}}{a^2}\to 0.$ Keeping in mind to apply the dominated convergence theorem for $|z|>1,$ we estimate as follows. 

$$\left|\frac{\Re \{b\}}{a^2}\right|\leq |x|, \hspace{0.1in} 1 < |z| < 2 $$ 

$$\left|\frac{\Re \{b\}}{a^2}\right|\leq|x|\frac{2n}{|z|^{2n+2}}\leq\frac{2}{|z|^3}, \hspace{0.1in} |z| > 2 $$

From
\begin{equation}\label{eq:Sigma}
\frac{|\Sigma|}{a^2} = \partial_{\bar{z}} \partial_z \log K(z,z) = 
\frac{1}{(1-|z|^2)^2}- \frac{(n+1)^2|z|^{2n}}{(1-|z|^{2n+2})^2},
\end{equation}
we notice that for $|z| < 1,$
\begin{align*}
\frac{|\Sigma|}{a^3} &= \frac{\partial_{\bar{z}} \partial_z \log K(z,z)}{K(z,z)} \\
&= \frac{1}{1-|z|^2}\left(\frac{1-\frac{(n+1)^2|z|^{2n}}{(\sum_{k=0}^n |z|^{2k})^2}}{1-|z|^{2n+2}}   \right)\\
&\leq \frac{1}{1-|z|^2},
\end{align*}
\noindent and  $\frac{|\Sigma|}{a^3}\to \frac{1}{1-|z|^2}$ as $n\to\infty.$

\vspace{0.1in}

\noindent A similar computation for $|z| > 1,$ yields

$$ \frac{|\Sigma|}{a^3} = \frac{1}{|z|^2-1}\left(\frac{1-\frac{(n+1)^2|z|^{2n}}{(\sum_{k=0}^n |z|^{2k})^2}}{|z|^{2n+2}-1}\right)\leq \frac{1}{|z|^2-1},$$

\noindent and as $n\to\infty,$ $\frac{|\Sigma|}{a^3}\to 0.$ To apply dominated convergence, we use the following bounds which follow immediately from the above expression

$$\left|\frac{|\Sigma|}{a^3}\right|\leq \frac{1}{|z|^2-1}, \hspace{0.1in} 1 < |z| < 2.$$

$$\left|\frac{|\Sigma|}{a^3}\right|\leq \frac{2}{|z|^{2n+2}}\leq\frac{2}{|z|^{6}}, \hspace{0.1in} |z| > 2, n\geq 2.$$

Letting $F_n(z)$ denote the integrand in \eqref{eq:nonasymp}, we have:
\begin{align*}\label{eq:dominated}
 F_n(z) &= \frac{\exp \{-1/a \}}{a} \left[ \sqrt{\frac{|\Sigma|}{a}}\exp\left\{-\frac{|\Re{b}|^2}{a|\Sigma|} \right\} + \sqrt{\pi} \frac{|\Re b|}{a} \erf \left\{ |\Re b|/\sqrt{a |\Sigma|} \right\} \right] \\
 &\leq \exp\{-1/a \} \left[ \sqrt{\frac{|\Sigma|}{a^3}} + \sqrt{\pi} \frac{|\Re b|}{a^2} \right] .\\
 \end{align*}
For $|z|<1$ we have:
$$F_n(z) \leq \exp\left\{-{(1-|z|^2)} \right\} \left[ \frac{1}{\sqrt{1-|z|^2}} + \sqrt{\pi} x \right] ,$$
which is integrable. If $|z|> 1$ and $n$ is large enough, we split the integral into regions $ 1 < |z| < 2$ and $|z| >2$ and use the appropriate bounds from before. This justifies the use of the dominated convergence theorem.

In order to see the pointwise limit \eqref{eq:pointwise} of $F_n(z)$,
we notice that for $|z| < 1$, we have (as $n \rightarrow \infty$):
$$ \sqrt{\frac{|\Sigma|}{a^3}} \rightarrow \frac{1}{\sqrt{1-|z|^2}} ,$$
$$ a \rightarrow \frac{1}{1-|z|^2},$$
$$ \frac{\Re\{ b \}}{a^2} \rightarrow x,$$
and
$$ \frac{\Re\{ b \}}{\sqrt{a |\Sigma|}} \rightarrow x \sqrt{1-|z|^2}.$$
As pointed earlier, for $|z|> 1$, we have:
$$F_n(z)\rightarrow 0.$$

Combining these pointwise limits, we arrive at \eqref{eq:pointwise},
and applying the dominated convergence theorem proves 
the formula \eqref{eq:limit} for the asymptotic expected length of a lemniscate
generated by the Kac model.


\section{The expected length for other models}

\subsection{Kostlan Polynomials}\label{sec:Kostlan}

\noindent In this section we compare the average length of the lemniscate for different ensembles of random polynomials, starting with the Kostlan ensemble.

\noindent Consider a sequence of random polynomials whose coefficients are Kostlan random variables. Namely

$$P_n(z) = \sum_{k=0}^{n}a_{kn}z^k,$$

\noindent where $a_{kn}$ are independent $N_{\mathbb C}(0, \binom{n}{k}).$ Applying Theorem \ref{thm:general}

\begin{equation}\label{nonasymp2}
\EE |\Lambda | =  \sqrt{\pi} \int_{\C} \frac{\exp \{-\frac{1}{a} \}}{a} \left[ \sqrt{\frac{|\Sigma|}{a}}\exp\left\{-\frac{|\Re{b}|^2}{a|\Sigma|} \right\} + \sqrt{\pi} \frac{|\Re b|}{a} \erf \left\{ |\Re b|/\sqrt{a |\Sigma|} \right\} \right] dA(z).
\end{equation} 

\noindent where now for the Kostlan ensemble,

$$ \Sigma = \left( \begin{array}{cr}
a & b \\ \bar{b} & c 
\end{array} \right) ,$$
with
\begin{align*}
a &= K(z,z) = (1+|z|^2)^{n},\\
b &= n\bar{z}(1+|z|^2)^{n-1} ,\\
c &= n(n|z|^2+1)(1+|z|^2)^{n-2}.
\end{align*}

\noindent This implies that 
\begin{align*}
|\Sigma|& = ac - |b|^2 = n(1+|z|^2)^{2n-2}, \\
\frac{|\Sigma|}{a^3} & = \dfrac{n}{(1+|z|^2)^{n+2}},\\
\frac{\Re \{b\}}{a^2} & = \dfrac{nx}{(1+|z|^2)^{n+1}}, \\
\frac{|\Re b|}{\sqrt{a |\Sigma|}} & = \dfrac{nx^2}{(1+|z|^2)^{n}}.
\end{align*}

\noindent Substituting these expressions into \eqref{nonasymp2}, we obtain

\begin{equation}\label{nonasymp3}
\EE |\Lambda_n | = \sqrt{\pi}\int_{\mathbb{C}}\exp\left(-\frac{1}{(1+|z|^2)^{n}}\right)\left[I_{1n}(z) + I_{2n}(z)\right] dA(z)
\end{equation}

\noindent where $I_{1n}(z) = \sqrt{\frac{n}{(1+|z|^2)^{n+2}}}\exp\left(-\frac{nx^2}{(1+|z|^2)^{n}}\right)$ and $I_{2n}(z)= \sqrt{\pi}\frac{nx}{1+|z|^2} \erf \left\{\sqrt{n}x/(1+|z|^2)^{n/2}\right\}$

\noindent Converting the above integral into polar coordinates $(r, \theta),$ followed by the substitution $r = \sqrt{\frac{t}{n}}$ leads us to

\begin{equation}\label{t}
\EE |\Lambda_n | = \sqrt{\pi}\int_{0}^{2\pi}\int_{0}^{\infty}\exp\left(-\frac{1}{(1+t/n)^{n}}\right)\left[J_{1n}(t, \theta) + J_{2n}(t, \theta)\right]dtd\theta,
\end{equation}

$$J_{1n}(t, \theta)= \sqrt{\frac{1}{n(1+t/n)^{n+2}}}\exp\left(-\frac{t\cos^2(\theta)}{(1+ t/n)^{n}}\right)$$ 

$$J_{2n}(t, \theta) = \sqrt{\pi} \frac{\sqrt{t}\cos(\theta)}{\sqrt{n}(1+t/n)} \erf \left\{\sqrt{t}\cos(\theta)/(1+t/n)^{n/2}\right\}.$$

\noindent Removing a factor of $1/\sqrt{n}$ from the $J_{in}$, we see that the resulting integral has a limit as $n\to\infty.$ Namely, we have the following result

$$\sqrt{n}\EE |\Lambda_n |\to I \hspace{0.1in}\mbox{as}\hspace{0.1in} n\to\infty,$$

where $I$ is the constant given by
$$I = \sqrt{\pi}\int_{0}^{2\pi}\int_{0}^{\infty}\exp\left(-\frac{1}{e^t}\right)\left[\sqrt{\frac{1}{e^t}}\exp\left(-\frac{t\cos^2(\theta)}{e^t}\right)+ \sqrt{\pi}\sqrt{t}\cos(\theta)\erf \left\{\sqrt{t}\cos(\theta)/e^{t/2}\right\}\right]dtd\theta.$$

\subsection{Weyl Polynomials}\label{sec:Weyl}

We now consider Weyl polynomials defined by $P_n(z) = \sum_{k=0}^{n}a_kz^k$ where $a_k$ are independent random variables with $a_k \sim N_{\mathbb{C}}(0, \frac{1}{k!}).$ 

\vspace{0.1in}

\noindent One can check easily that now the covariance matrix has entries given by

$$ \Sigma = \left( \begin{array}{cr}
a & b \\ \bar{b} & c 
\end{array} \right) ,$$

\noindent with

\begin{align*}
a & = \sum_{k=0}^{n}|z|^{2k}/k!,\\
b & = \bar{z}\sum_{k=1}^{n}|z|^{2k-2}/(k-1)! ,\\
c & = \sum_{k=1}^{n}\frac{k^2}{k!}|z|^{2k-2}.
\end{align*}

\noindent Applying Theorem \ref{thm:general}, we obtain

\begin{equation}\label{nonasymp4}
\EE |\Lambda_n | =  \sqrt{\pi} \int_{\C} \frac{\exp \{-\frac{1}{a} \}}{a} \left[ \sqrt{\frac{|\Sigma|}{a}}\exp\left\{-\frac{|\Re{b}|^2}{a|\Sigma|} \right\} + \sqrt{\pi} \frac{|\Re b|}{a} \erf \left\{ |\Re b|/\sqrt{a |\Sigma|} \right\} \right] dA(z).
\end{equation} 

\noindent All the quantities above have finite limits as $n\to\infty.$ For instance $a\to\exp(|z|^2),$ $b\to\bar{z}\exp(|z|^2),$ and $c\to (1+|z|^2)\exp(|z|^2).$ Also,  dominated convergence is easy to verify here. Taking the limit as $n\to\infty$ in \eqref{nonasymp4}, we obtain

$$\EE |\Lambda_n |\to L,$$

where $$L = \sqrt{\pi} \int_{\C} \frac{\exp \{-\frac{1}{e^{|z|^2}} \}}{e^{|z|^2}} \left[ \sqrt{e^{|z|^2}}\exp\left\{-\frac{x^2}{e^{|z|^2}} \right\} + \sqrt{\pi}x\erf \left\{x/e^{|z|^2/2} \right\} \right] dA(z).$$

\subsection{Reciprocal binomial distribution}\label{sec:reciprocal}

\noindent Consider a random polynomial of the form 
$$p_n(z) = \sum_{k=0}^{n}a_{nk}z^k,$$
where $a_{nk}$ are independent random variables with 
$a_{nk} \sim N_{\mathbb{C}} \left( 0, \frac{1}{\binom{n}{k}} \right).$

In this case, the entries $a, b$ and $c$ the 
entries of the covariance matrix $\Sigma$ 
are given as follows. 

\begin{align*}
a & = \sum_{k=0}^{n}\frac{|z|^{2k}}{\binom{n}{k}},\\
b & = \bar{z}\sum_{k=1}^{n}\frac{k|z|^{2k-2}}{\binom{n}{k}} ,\\
c & = \sum_{k=1}^{n}\frac{k^2}{\binom{n}{k}}|z|^{2k-2}.
\end{align*}

\noindent Theorem \ref{thm:general} gives,

\begin{equation}\label{nonasymp5}
\EE |\Lambda_n | =  \sqrt{\pi} \int_{\C} \frac{\exp \{-\frac{1}{a} \}}{a} \left[ \sqrt{\frac{|\Sigma|}{a}}\exp\left\{-\frac{|\Re{b}|^2}{a|\Sigma|} \right\} + \sqrt{\pi} \frac{|\Re b|}{a} \erf \left\{ |\Re b|/\sqrt{a |\Sigma|} \right\} \right] dA(z).
\end{equation} 

We consider now asymptotically (with $n$) the contribution of this integral from $|z| <1$ and $|z| > 1$. If $|z| <1,$ then we observe from the expressions for $a, b$ and $c$ that

\begin{align*}
a & = 1 + \frac{|z|^2}{n} + o(1),\\
b & = \frac{\bar{z}}{n}\left(1 + \frac{4}{n-1}|z|^2 + o(1) \right) ,\\
c & = \frac{1}{n}\left(1 + \frac{8}{n-1}|z|^2 + o(1) \right).
\end{align*}

\noindent This yields $|\Sigma| = ac - |b|^2 = \frac{1}{n}\left(1+o(1)\right),$  $\frac{|\Re b|}{a} = \frac{x}{n}(1+o(1))$ and finally $\frac{|\Re{b}|^2}{a|\Sigma|} = \frac{x^2}{n}(1+o(1)).$ This implies that the integral for $|z| < 1$ is of order $\sqrt{\frac{1}{n}}\left(1+o(1)\right).$ So $\sqrt{n}\EE |\Lambda_n |$ has a finite limit for $z$ in the unit disc.

\vspace{0.1in}

\noindent We next claim that asymptotically, the integral over $|z| > 1$ goes to $0$ (after a scaling by $\sqrt{n}$). Indeed, notice then that 

\begin{align*}
a & = |z|^{2n}\left(1+ \frac{1}{n|z|^2} + o(1) \right),\\
b & = n\bar{z}|z|^{2n-2}\left(1 +  \frac{n-1}{n^2|z|^2}+ o(1) \right) ,\\
c & = n^2|z|^{2n-2}\left(1 + \frac{(n-1)^2}{n^3|z|^2}+ o(1) \right).
\end{align*}.

\noindent From here we can deduce that $|\Sigma| = \frac{|z|^{4n-4}}{n}(1+o(1)).$ This gives that 

$$\sqrt{\frac{|\Sigma|}{a^3}} = \sqrt{\frac{1}{n}}\frac{1}{|z|^{n+2}}(1+o(1)),$$

$$\frac{|\Re b|}{a^2} = \frac{n|x|}{|z|^{2n+2}}(1 + o(1)).$$

\noindent The pointwise limit of the integrand (even if we scale it by $\sqrt{n}$) is clearly $0$ and because of power decay, dominated convergence holds. So the contribution from the exterior of the unit disc to the integral is negligible. 
Ultimately, as $n \rightarrow \infty$, we get that $\sqrt{n}\EE |\Lambda_n |$ approaches
a positive constant given by an integral over $|z| < 1$ independent of $n.$

\vspace{0.1in}

\section{The connected components of a random lemniscate}\label{sec:components}
\noindent In this section, 
we prove asymptotics for the expected number of 
connected components $\mathbb{E}(b_0(\Lambda_n))$
of a lemniscate $\Lambda_n = \{z: |p_n(z)| = 1\}$,
where $p_n$ is a random Kac polynomial, i.e., 
$p_n(z) = \sum_{k=0}^{n}a_kz^k,$ 
with i.i.d. coefficients $a_k \sim N_{\mathbb{C}}(0,1)$.

Consider the set:
\begin{equation}\label{lemn1}
U_n = \{z: |p_n(z)| < 1\}.
\end{equation}
\noindent 
Then $U_n$ is a bounded open set and it 
is a well-known fact that 
the number of connected components of 
$U_n$ is at most $n.$ 
This can be seen from noticing that each component
of $U_n$ must contain a zero of $p$.
Otherwise the maximum principle may be applied to conclude that
the harmonic function $\log|p|$ is constant.
It also follows from the maximum principle that each
component of $U_n$ is simply-connected.
The boundary of $U_n$ is the lemniscate $\Lambda_n$,
which is smooth with probability one.
We conclude that the connected components of
$\Lambda_n$ are in one-to-one correspondence with those of $U_n$.

\subsection{The expectation of the number of connected components: proof of Theorem \ref{thm:cc}}
\label{sec:cc}
Since the number of connected components $b_0(\Lambda_n)$ 
is at most $n$, in order to show that $\E b_0(\Lambda_n) \sim n$
it suffices to prove the lower bound 
$\mathbb{E} b_0(\Lambda_n) \geq n - o(n)$.

Fix $0< \beta < \alpha <1/2$ with $\alpha - \beta > \frac{1}{2} -\alpha$,
and suppose $n$ is large enough that
$$n^{\beta+\frac{1}{2}- 2 \alpha} \exp\{n^{-\alpha}\} < 1.$$
As a certificate for the appearance of a localized component
we will use the following conditions related to the
Taylor expansion of $p(z)$ centered at $\zeta$.
\begin{equation}
\left\{
\begin{aligned}
p(\zeta) &= 0 \\
|p'(\zeta)| &> 2 \cdot n^{1+\alpha} \\
|p^{(k)}(\zeta)| &< n^{k+\frac{1}{2}+\beta} , \quad \text{for } k=2,3,..,n 
\end{aligned}\right.
\label{eq:Taylor}
\end{equation}
These conditions imply that, for any $z$ on the circle defined by
$|z-\zeta| = n^{-1-\alpha}$, we have
\begin{align}
|p(z)| &= \left| p'(\zeta) (z-\zeta) + \sum_{k=2}^n \frac{p^{(k)}(\zeta)}{k!}(z-\zeta)^k \right| \\
&\geq |p'(\zeta) (z-\zeta)| - \left|\sum_{k=2}^n \frac{p^{(k)}(\zeta)}{k!}(z-\zeta)^k \right| \\
&\geq |p'(\zeta)| n^{-1-\alpha} - \sum_{k=2}^n \frac{|p^{(k)}(\zeta)|}{k!}(n^{-1-\alpha})^k \\
&> 2 - \sum_{k=2}^n \frac{n^{(k+\frac{1}{2}+\beta)}}{k!}(n^{-1-\alpha})^k \\
&> 2 - n^{\beta + \frac{1}{2} - 2\alpha} \sum_{k=2}^n \frac{n^{-\alpha (k-2)}}{k!} \\
&> 2 - n^{\beta + \frac{1}{2} - 2\alpha} \exp \{ n^{-\alpha} \} \\
&> 1, \\
\end{align}
so that $p(\zeta)=0$ and $|p(z)|>1$ 
on the circle $|z-\zeta| = n^{-1-\alpha}$.
This ensures that there is a connected component of $\Lambda_n$
contained in the disk $|z-\zeta| < n^{-1-\alpha}$.

In order to estimate the average number of zeros
for which the conditions \eqref{eq:Taylor} are all satisfied,
we will use a modified version of the Kac-Rice formula.
First recall that the Kac-Rice formula for the 
expectation $\E N_p(U)$ of the number of complex zeros of $p$
in a region $U$ states
\begin{align}\label{eq:KRplain}
\E N_p(U) &= \frac{1}{\pi} \int_U \E |p'(z)|^2 \delta(p(z)) dA(z) \\
&= \frac{1}{\pi} \int_U \E \left[ |p'(z)|^2 \, \big| \, p(z)=0 \right] \rho_{p(z)} (0) dA(z),
\end{align}
where $\rho_{p(z)} (0)$ is the marginal probability density of $p(z)$ evaluated at $0$.

We would like to modify \eqref{eq:KRplain} to obtain a lower bound for the expected number $\hat{N}_p$
of zeros satisfying the conditions \eqref{eq:Taylor}. Our approach is based on \cite{AD}, Theorem $5.1.1$. Let $I_1$ be the indicator function of the interval $(2n^{1+\alpha},\infty)$
and $I_k$ be the indicator function of the interval $[0,n^{k+\frac{1}{2}+\beta}).$ Let $T_n(s) := \{z \in \C : e^{-s/n} < |z| < e^{s/n} \}$ and $\hat{N}_p(T_n(s))$ denote the number of zeros satisfying \eqref{eq:Taylor} which lie in the annulus $T_n(s).$
Then we have
\begin{align*}
\E \hat{N}_p &\geq \E \hat{N}_p(T_n(s)) \\
&= \frac{1}{\pi} \int_{T_n(s)} \E |p'(z)|^2 \delta(p(z))
\prod_{k=1}^n I_k(|p^{(k)}(z)|) dA(z) \\
&= \frac{1}{\pi} \int_{T_n(s)}  \E \left[ |p'(z)|^2 
\prod_{k=1}^n I_k(|p^{(k)}(z)|)  \, \big| \, p(z)=0 \right] \rho_{p(z)}(0)dA(z).
\end{align*}

In the above chain, Theorem $5.1.1$ from \cite{AD} was used to go from the first line to the second. Next, for each fixed $s$ the above provides a lower bound on the
average number of connected components
\begin{equation}\label{eq:pivot}
\E b_0(\Lambda_n) \geq \frac{1}{\pi} \int_{T_n(s)}  \E \left[ |p'(z)|^2 \prod_{k=1}^n I_k(|p^{(k)}(z)|)  \, \big| \, p(z)=0 \right] \rho_{p(z)}(0)dA(z).
\end{equation}
The remainder of the proof will establish that the right hand side
of \eqref{eq:pivot}
is asymptotic to a standard Kac-Rice integral of the form \eqref{eq:KRplain}.

Letting $\tilde{I}_k$ denote the indicator function of
$[n^{k+\frac{1}{2}+\beta},\infty)$,
we will use the union-type bound,
\begin{equation}\label{eq:union}
\prod_{k=2}^nI_k(|p^{(k)}(z)|) \geq
1-\sum_{k=2}^n\tilde{I}_k(|p^{(k)}(z)|),
\end{equation}
in order to prove that
\begin{equation}\label{eq:experr}
\E \left[ |p'(z)|^2 
\prod_{k=1}^n I_k(|p^{(k)}(z)|)  \, \big| \, p(z)=0 \right]
\geq \E \left[ |p'(z)|^2 I_1(|p'(z)|) \, \big| \, p(z)=0 \right] - O\left(\exp \left\{-n^{\beta}  \right\} \right) . \\
\end{equation}
First, we use the simple estimate:
\begin{equation}\label{eq:leftover}
\E \left[ |p'(z)|^2 I_1(|p'(z)|) \sum_{k=2}^n \tilde{I}_k(|p^{(k)|}(z)) \, \big| \, p(z)=0 \right] 
\leq \sum_{k=2}^n \E \left[ |p'(z)|^2 \tilde{I}_k(|p^{(k)}(z)|) \, \big| \, p(z)=0 \right].\\
\end{equation}
We estimate each summand above using the Cauchy-Schwarz inequality.
\begin{equation}\label{eq:apply}
\begin{aligned}
\E \left[ |p'(z)|^2 \tilde{I}_k(|p^{(k)}(z)|) \, \big| \, p(z)=0 \right]
&\leq \sqrt{\E \left[ |p'(z)|^4 \, \big| \, p(z)=0 \right] } \sqrt{P(|p^{(k)}(z)| \geq n^{k+\frac{1}{2}+\beta} \big| p(z)=0)} \\
&\leq \sqrt{\E \left[ |p'(z)|^4 \, \big| \, p(z)=0 \right] } \exp \left\{-\frac{n^{2\beta}}{2C_1(s)}  \right\},
\end{aligned}
\end{equation}
where we have used the estimates
\begin{equation}\label{eq:key}
P( |p^{(k)}(z)| \geq n^{k+\frac{1}{2} + \beta} \big| p(z)=0) \leq \exp \left\{-\frac{n^{2k+1+2\beta}}{C_1(s)n^{2k+1}}  \right\} = \exp\left\{ -\frac{n^{2\beta}}{C_1(s)} \right\},
\end{equation}
which follow from 
Lemmas \ref{lemma:conditionalGaussian} and \ref{lemma:overwhelming} below.
By the same lemmas, we have 
$\sqrt{\E \left[ |p'(z)|^4 \, \big| \, p(z)=0 \right] } = O(n^3)$. 

Applying \eqref{eq:apply} to \eqref{eq:leftover}
and relaxing the expression appearing in the exponent to $-n^{\beta}$,
we can neglect the polynomially growing factor 
$\sqrt{\E \left[ |p'(z)|^4 \, \big| \, p(z)=0 \right] } = O(n^3)$
as well as the number of terms $(n-1)$ in the sum.
We thus obtain the bound
\begin{equation}
\E \left[ |p'(z)|^2 I_1(|p'(z)|) \sum_{k=2}^n \tilde{I}_k(|p^{(k)|}(z)) \, \big| \, p(z)=0 \right] 
= O\left(\exp \left\{-n^{\beta}  \right\} \right),
\end{equation}
which establishes \eqref{eq:experr}
by way of the union bound stated in \eqref{eq:union}.

The random variable $p'(z)$
conditioned on $p(z)=0$ is distributed 
as a centered complex Gaussian with variance $\frac{ac-|b|^2}{a}$,
and this implies that $|p'(z)|^2$ conditioned on $p(z)=0$
is distributed as an exponential random variable with parameter 
$\lambda = \left( \frac{ac_1-|b_1|^2}{a} \right)^{-1}$,
so we have
\begin{equation}\label{eq:plaincond}
 \E \left[ |p'(z)|^2 \, \big| \, p(z)=0 \right] 
 = \frac{1}{\lambda} = \frac{ac_1-|b_1|^2}{a},
\end{equation}
and
\begin{align*}
\E \left[ |p'(z)|^2 I_1(|p'(z)|) \, \big| \, p(z)=0 \right] 
&= \int_{4n^{2+2\alpha}}^\infty x \lambda e^{-\lambda x} dx \\
&=  \exp \left\{-n^{2\alpha+2} \lambda \right\} \left( \frac{1}{\lambda} + 4n^{2+2\alpha} \right)\\
&\geq \frac{1}{\lambda} \exp \left\{-n^{2\alpha+2} \lambda \right\}\\
&=\E \left[ |p'(z)|^2 \, \big| \, p(z)=0 \right] \left(1 - O(n^{2\alpha-1}) \right) \end{align*}
where we used \eqref{eq:plaincond} in the last line.

Combining this with \eqref{eq:experr}
in order to reassess \eqref{eq:pivot}
we finally conclude the lower bound
\begin{equation}
\begin{aligned}
\E b_0(\Lambda_n) &\geq 
(1-O(n^{2\alpha-1})) \frac{1}{\pi} \int_{T_n(s)}  \E \left[ |p'(z)|^2 
\, \big| \, p(z)=0 \right] \rho_{p(z)}(0)dA(z), \\
&= (1-O(n^{2\alpha-1})) \E N_p(T_n(s)),
\end{aligned}
\end{equation}
where, as in \eqref{eq:KRplain} 
$N_p(T_n(s))$ denotes 
the number of zeros of $p$ in $T_n(s)$.
We recall \cite{IbZeit} that 
$$\E N_p(T_n(s)) \sim n\left(\frac{1+e^{2s}}{1-e^{2s}} - \frac{1}{s}\right),$$
which implies
\begin{equation}
\liminf_{n \rightarrow \infty} \frac{\E b_0(\Lambda_n)}{n}
\geq \left( 1 - \frac{1}{s} \right).
\end{equation}
This lower bound can be made arbitrarily close to $1$
(by increasing $s$),
and along with the deterministic upper bound
$b_0(\Lambda_n) \leq n$ this shows that the limit
$$\lim_{n \rightarrow \infty} \frac{\E b_0(\Lambda_n)}{n} = 1$$
exists, i.e., $\E b_0(\Lambda_n) \sim n$.
This proves Theorem \ref{thm:cc}.

\begin{lemma}\label{lemma:conditionalGaussian}
Fix $z \in\mathbb{C}$.
The random variable $p^{(k)}(z)$ conditioned on
$p(z)=0$ is distributed as a centered complex Gaussian,
$N_\C(0,\sigma^2)$,
with variance 
$$\sigma^2 = \frac{a c_k - |b_k|^2}{a},$$
where
\begin{equation}
a = K(z,z), \quad b_k = \partial_{z}^k K(z,z), \quad c_k = \partial_{z}^k \partial_{\bar{z}}^k K(z,z).
\end{equation}
\end{lemma}

\begin{proof}[Proof of Lemma \ref{lemma:conditionalGaussian}]
Let $\rho(u,v)$ denote the joint density of $(U,V) = (p(\zeta),p^{(k)}(\zeta))$.
The conditional density $\rho_{V|U=0}$ of $V$ given $U=0$
is given by:
\begin{equation}\label{eq:conditdensity}
\rho_{V|U=0} (v) = \frac{\rho(0,v)}{\rho_U(0)},
\end{equation}
where $\rho_U(u) = \frac{1}{\pi a} \exp\left\{-\frac{|u|^2}{a} \right\}$ 
is the marginal density of $U$.

We have
\begin{equation}
\rho(u,v) = \frac{1}{\pi^2 |\Sigma_k|} \exp \{- (u,v)^* \Sigma_k^{-1} (u,v) \},
\end{equation}
where $\Sigma_k$ is the covariance matrix of $(U,V)$,
which can be computed explicitly using the covariance kernel $K(z,w)$:
$$K(z,w) = \EE p(z) \ol{p(w)}.$$
Namely, we have:
$$ \Sigma_k = \left( \begin{array}{cr}
a & b_k \\ \bar{b_k} & c_k
\end{array} \right) ,$$
where
\begin{equation}
a = K(z,z), \quad b_k = \partial_{z}^k K(z,z), \quad c_k = \partial_{z}^k \partial_{\bar{z}}^k K(z,z).
\end{equation}
Applying this to \eqref{eq:conditdensity} we obtain:
\begin{equation}
\rho_{V|U=0} (v) = \frac{\rho(0,v)}{\rho_U(0)} = \frac{a}{\pi |\Sigma_k|} \exp \left\{- \frac{a|v|^2}{|\Sigma_k|} \right\},
\end{equation}
as desired.
\end{proof}

\begin{lemma}\label{lemma:overwhelming}
There exists a positive constant $C_1(s)$ depending on $s$ but independent of $n$, such that

$$ \frac{ac_k - |b_k|^2}{a} \leq C_1(s) n^{2k+1},$$
for all $z \in T_n(s) = \{z\in\C: e^{-s/n} < |z| < e^{s/n} \}$ and $k=1,2,..,n$. Furthermore, there exists $C_2(s)>0$ such that for $z \in T_n(s),$ and for all large enough $n\geq N(s),$ we have

$$ \frac{ac_1-|b_1|^2}{a} \geq C_2(s) n^{3}.$$
\end{lemma}

\begin{proof}[Proof of Lemma \ref{lemma:overwhelming}]
\noindent For the first estimate, we note that $\dfrac{ac_k - |b_k|^2 }{a}\leq c_k$ and so it is enough to find an upper bound for $c_k.$ We have

$$c_k = \mathbb{E}\left(p^{(k)}(z)\overline{p^{(k)}(z)}\right) = \sum_{j=k}^{n}\left[j(j-1)(j-2)..(j-(k-1)\right]^2|z|^{2j-2k}$$

\noindent Using the above expression, we observe that for $z\in T_n(s),$

\begin{equation}\label{ubc}
c_k\leq e^{2s}\sum_{j=k}^{n}\left[j(j-1)(j-2)..(j-(k-1)\right]^2\leq e^{2s}\sum_{j=k}^{n} n^{2k}\leq e^{2s}n^{2k+1}.
\end{equation}

\noindent Before proving the second inequality we recall that 
$$a = \sum_{k=0}^n |z|^{2k}, \hspace{0.05in} b_1 = \bar{z} \sum_{k=1}^n k |z|^{2k-2}$$

$$c_1 = \sum_{k=1}^n k^2 |z|^{2k-2}.$$
\noindent For $z \in T_n(s),$ we now estimate as follows:

\begin{equation}\label{lba}
a = \sum_{k=0}^n |z|^{2k}\geq (e^{-s/n})^{2n}(n+1) = e^{-2s}(n+1).
\end{equation}

\noindent  A similar reasoning gives $a\leq (n+1)e^{2s}.$  
We next proceed to bound $c_1$ and $|b_1|^2.$ 
\begin{equation}\label{lbc}
c_1 = \sum_{k=1}^n k^2 |z|^{2k-2}\geq (e^{-s/n})^{2n-2}\sum_{k=1}^n k^2 = (e^{-s/n})^{2n-2}n(n+1)(2n+1)/6
\end{equation}

\begin{eqnarray}\label{ubb}
|b_1|^2& = |z|^2\left(\sum_{k=1}^n k |z|^{2k-2}\right)^2\\
     &\leq e^{2s/n}e^{4s}\left(\sum_{k=1}^{n} k\right)^2 \\
     & =  e^{2s/n}e^{4s}\dfrac{n^2 (n+1)^2}{4}.
\end{eqnarray}

\noindent Combining all the above estimates, we obtain that for large $n$ 

$$\dfrac{ac_1 - |b|^2}{a}\geq C_2(s)n^3.$$
\noindent This proves the second estimate and concludes the proof of the lemma.

 
\end{proof}

\subsection{Existence of a giant component: proof of Theorem \ref{thm:giant}}
\label{sec:giant}
We now show that a giant component exists 
with positive probability (independent of $n$).

\begin{lemma}\label{giant}
Consider a sequence of random polynomials $p_n(z) = \sum_{k=0}^{n}a_kz^k,$ where $a_k$ are i.i.d $\sim N_{\mathbb{C}}(0,1).$ Let $U_n$ be as in \eqref{lemn1} and let $r\in (0,1)$ be given. 
Then, there exist $N = N(r)\in\mathbb{N}$ and $c_r > 0$ such that for all $n\geq N$
\begin{equation}\label{eq:giantdomain}
\mathbb{P}\left(B(0,r)\hspace{0.05in}\mbox{is contained in a component of $U_n$}\hspace{0.05in}\right) > c_r .
\end{equation}
\end{lemma}
\begin{proof}
For each $r\in (0,1),$ consider $g(r) = \sum_{k=0}^{\infty}|a_k|r^k.$ Then $g$ is a random function and $\mathbb{E}(g(r)) < \infty.$ Therefore, there exist $a_r, b_r > 0$ such that 

$$\mathbb{P}\left(g(r) < a_r\right) > b_r > 0.$$

\noindent For a given $r\in (0,1)$ choose $N$ so that $r^N < \frac{1}{2a_r}.$ Then, for $n\geq N$

\begin{align*}
\mathbb{P}\left(\sup_{\partial B_r}|p_n| < 1 \right)& \geq \mathbb{P}\left(|a_0| + |a_1|r +... |a_{N-1}|r^{N-1} < \frac{1}{2} ; r^N\sum_{j=N}^n|a_j|r^{j-N} < \frac{1}{2}\right), \\
&\geq\mathbb{P}\left(|a_0| + |a_1|r +... |a_{N-1}|r^{N-1} < \frac{1}{2}\right)\mathbb{P}\left(r^Ng(r) <\frac{1}{2} \right) \\
& \geq\eta_{r}\mathbb{P}\left(g(r) < a_r\right)\\
& = \eta_{r} b_r,
\end{align*}
where $\eta_{r} = \mathbb{P}\left(|a_0| + |a_1|r +... |a_{N-1}|r^{N-1} < \frac{1}{2}\right) > 0$ follows from the Gaussian nature of the coefficients. Note that we have used the independence of 
$a_k$'s to go from the first line to the second. This finishes the proof of the Lemma.
\end{proof}

In the case the event in \eqref{eq:giantdomain} occurs, 
by the isoperimetric inequality,
the associated connected component of $\Lambda_n$
has length at least $2\pi r$.
This proves Theorem \ref{thm:giant}.

\vspace{0.1in}

\noindent \textbf{Acknowledgements:} The authors are grateful to Antonio Lerario and Manjunath Krishnapur for helpful discussions and suggestions.

\vspace{0.1in}
 
{\em 

Department of Mathematical Sciences

Florida Atlantic University

Boca Raton, FL 33431

Email: elundber{@}fau.edu}

\vspace{0.1in}

\em

Department of Mathematics

Oklahoma State University

Stillwater, OK 74074

Email: koushik.ramachandran@okstate.edu

\end{document}